\documentclass[12pt, reqno]{amsart}
\usepackage[utf8]{inputenc}
\usepackage{amsmath}
\usepackage{amsthm}
\usepackage{amsfonts}
\usepackage{amssymb}
\usepackage{blindtext,amsmath,amsthm,amsfonts,amssymb, textcomp, mathrsfs, mathtools}

\usepackage[top=20mm,bottom=20mm,left=22mm,right=22mm]{geometry}
\newtheorem{theorem}{Theorem}[section]

\newtheorem*{theorem*}{Theorem}
\newtheorem*{remark*}{Remark}
\newtheorem{remark}{Remark}
\newtheorem*{problem*}{Problem}
\newtheorem*{conjecture*}{Conjecture}

\newtheorem{corollary}[theorem]{Corollary}
\newtheorem{prop}[theorem]{Proposition}
\newtheorem{definition}[theorem]{Definition}

\begin{document}

\baselineskip=17pt

\title[Value distribution of $L$-functions]{Value distribution of $L$-functions}

\author[Anup B. Dixit]{Anup B. Dixit}
\address{Department of Mathematics\\ University of Toronto\\ 40 St. George St\\ Canada, ON\\ M5S2E4}
\email{adixit@math.toronto.edu}

\begin{abstract}
In 2002, V. Kumar Murty \cite{Km} introduced a class of $L$-functions, namely the Lindel\"of class, which has a ring structure attached to it. In this paper, we establish some results on the value distribution of $L$-functions in this class. As a corollary, we also prove a uniqueness theorem in the Selberg class.
\end{abstract}



\date{}
\maketitle
\section{Introduction}

In 1992, Selberg \cite{Slb} formulated a class of $L$-functions, which can be regarded as a model for $L$-functions originating from arithmetic objects. The value distribution of such $L$-functions has been extensively studied  \cite{Std1} in recent times. The study of value distribution is concerned with the zeroes of $L$-functions and more generally, with the set of pre-images $L^{-1}(c):=\{s\in\mathbb{C}: L(s) = c\}$ where $c$ is any complex number, which Selberg called the $c$-values of $L$. 

For any two meromorphic functions $f$ and $g$, we say that they share a value $c$ ignoring multiplicity(IM) if $f^{-1}(c)$ is the same as $g^{-1}(c)$ as sets. We further say that $f$ and $g$ share a value $c$ counting multiplicity(CM) if the zeroes of $f(x)-c$ and $g(x)-c$ are the same with multiplicity.  The famous Nevanlinna theory \cite{Nev} establishes that any two meromorphic functions of finite order sharing five values IM must be the same. Moreover, if they share four values CM, then one must be a M\"obius transform of the other. The numbers four and five are the best possible for meromorphic functions. 

If one replaces meromorphic functions with $L$-functions, one can get much stronger results. In particular, it was shown by M. Ram Murty and V. Kumar Murty \cite{Rm} that if two $L$-functions in the Selberg class share a value CM, then they should be the same. Steuding \cite{Std2} further showed that two $L$-functions in the Selberg class sharing two values IM, with some additional conditions should be the same. In 2011, Bao Qin Li \cite{Bao2} proved the result of Steuding dropping the extra conditions. In a previous paper in 2010, Bao Qin Li \cite{Bao1} also showed that if $f$ is a meromorphic function  with finitely many poles and $L$ is an $L$-function from the extended Selberg class, such that they share one value CM and another value IM, then they should be the same.

In this paper, we establish all the above results in the more general setting of the Lindel\"of class, where we replace the functional equation and the Euler product in the Selberg class by a growth condition. This class has a rich algebraic structure and forms a differential graded ring. In particular, the Lindel\"of class is closed under derivatives, i.e., if $L$ is in the Lindel\"of class, then so is $L'$. We also show a different kind of uniqueness theorem which states that if a meromorphic function $f$ with finitely many poles and an $L$-function $L$ in the Lindel\"of class share a value CM and their derivatives share zeroes up to an error, then they must be the same. As a corollary, we get the same result for the Selberg class, which is a subset of the Lindel\"of class.

The article is organized as follows. In section \ref{class}, we introduce the class $\mathbb{M}_1$ of $L$-functions that we will be working with. In section \ref{nevanlinna}, we introduce notations from Nevanlinna theory. In section \ref{Main}, we state out main results and in sections \ref{preliminaries} and \ref{proofs}, we give the proofs.

\section{The class $\mathbb{M}_1$}\label{class}
The Selberg class \cite{Slb} $\mathbb{S}$ consists of functions $F(s)$ satisfying the following properties:\\\\
(1) Dirichlet series - It can be expressed as a Dirichlet series
\begin{equation*}
F(s) = \sum_{n=1}^{\infty} \frac{a_F(n)}{n^s},
\end{equation*}
which is absolutely convergent in the region $\Re(s)>1$.\\\\
(2) Analytic continuation - There exists a non-negative integer $k$, such that $(s-1)^k F(s)$ is an entire function of finite order.\\\\
(3) Functional equation - There exist non-negative real numbers $Q$ and $\alpha_i$, complex numbers $\beta_i$ with $\Re(\beta_i) \geq 0$, and $w\in\mathbb{C}$ satisfying $|w| =1$, such that
\begin{equation*}\label{fneq}
\Phi(s) := Q^s \prod_i \Gamma(\alpha_i s + \beta_i) F(s)
\end{equation*}
satisfies the functional equation
\begin{equation*}
\Phi(s) = w \bar{\Phi}(1-\bar{s}).
\end{equation*}
\\\\
(4) Euler product - There is an Euler product of the form
\begin{equation*}
F(s) = \prod_{p} F_p(s),
\end{equation*}
where
\begin{equation*}
\log F_p(s) = \sum_{k=1}^\infty \frac{b_{p^k}}{p^{ks}}
\end{equation*}
with $b_{p^k} = O(p^{k \theta})$ for a $\theta < 1/2$.\\\\
(5) Ramanujan Hypothesis - For any $\epsilon >0$, $|a_F(n)| = O_\epsilon(n^\epsilon)$.\\\\

The constants in the functional equation depend on $F$, and although the functional equation may not be unique, we have some invariants, such as the degree $d_F$ of $F$, defined by
\begin{equation*}
d_F = 2\sum_{i} \alpha_i.
\end{equation*}
The factor $Q$ in the functional equation gives rise to another invariant referred to as the conductor $q_F$, defined by
\begin{equation*}
q_F = (2\pi)^{d_F} Q^2 \prod_i {\alpha_i}^{2\alpha_i}.
\end{equation*}
These invariants play an important role in studying the growth of the $L$-function.

Note that the Selberg class is not closed under addition. In \cite{Km}, V. Kumar Murty defined a class of $L$-functions based on growth conditions. We start by defining two different growth parameters $\mu$ and $\mu*$.

\begin{definition}
{\bf The class $\mathbb{T}.$} We define the class $\mathbb{T}$ to be the set of functions $F(s)$ satisfying the following conditions:\\\\
(1) Dirichlet series - For $\sigma>1$, $F(s)$ is given by the absolutely convergent Dirichlet series
\begin{equation*}
F(s) = \sum_{n=1}^\infty \frac{a_F(n)}{n^s}.\\\\
\end{equation*}
(2) Analytic continuation - There exists a non-negative integer $k$, such that $(s-1)^k F(s)$ is an entire function of order $\leq 1$.\\\\
(3) Ramanujan Hypothesis - $|a_F(n)| = O_\epsilon(n^\epsilon)$ for any $\epsilon>0$.\\
\end{definition}

\begin{definition}  
For $F\in \mathbb{T}$, we define $\mu_F(\sigma)$ as
\begin{equation}\label{mu1}
\mu_F(\sigma) := \left \{ \begin{array}{ll}
                                      \inf \{ \lambda \geq 0 : |F(s)| \leq (|s| + 2)^\lambda, \text{ for all } s \text{ with } Re(s)=\sigma \} \\\\
                                      \infty, \text{ if the infimum does not exist.}
\end{array} \right.
\end{equation}
We also define:
\begin{equation}\label{mu2}
\mu_F^*(\sigma) := \left \{ \begin{array}{ll}
                                      \inf \{ \lambda \geq 0 : |F(\sigma + it)| \ll_{\sigma} (|t| + 2)^\lambda\}, \\\\
                                      \infty, \text{ if the infimum does not exist.}
\end{array} \right.
\end{equation}
with the implied constant depending on $F$ and $\sigma$. If $F\in\mathbb{T}$ has a pole of order $k$ at $s=1$, consider the function
\begin{equation*}
G(s) := \left( 1-\frac{2}{2^s} \right)^k F(s).
\end{equation*}
G(s) is an entire function and belongs to $\mathbb{T}$. We define
\begin{equation*}
\mu_F(\sigma) := \mu_G(\sigma).
\end{equation*}
\begin{equation*}
\mu_F^*(\sigma) := \mu_G^*(\sigma).
\end{equation*}
\end{definition}
Intuitively, $\mu_F^*$ does not depend on how $F$ behaves close to the real axis.\\

Since the Dirichlet series is absolutely convergent for $Re(s)>1$, we have $\mu_F^*(\sigma)=0$.\\

If $F\in\mathbb{S}$, by the functional equation\eqref{fneq}, using Stirling's formula, we have (see \cite[sec. 2.1]{Km})
\begin{equation}\label{deg}
\mu_F^*(\sigma) \leq \frac{1}{2} d_F (1-2\sigma) \text{ for } \sigma<0.
\end{equation}

Using the Phragm\'en-Lindel\"of theorem, we deduce that
\begin{equation*}
\mu_F^*(\sigma) \leq \frac{1}{2} d_F (1-\sigma) \text{ for } 0<\sigma<1.
\end{equation*}
The same results hold for $\mu_F$ upto a constant depending on $F$ by a similar argument.\\ 
It follows from the definition that $\mu_F^*(\sigma) \leq \mu_F(\sigma)$ for any $\sigma$.
\begin{definition} 
{\bf The class $\mathbb{M}$.} We now define the class $\mathbb{M}$ (\cite[sec.2.4]{Km}) to be the set of functions $F(s)$ satisfying the following conditions:\\\\
(1) Dirichlet series - For $\sigma>1$, $F(s)$ is given by the absolutely convergent Dirichlet series \begin{equation*}
\sum_{n=1}^\infty \frac{a_F(n)}{n^s}.\\\\
\end{equation*}
(2) Analytic continuation - There exists a non-negative integer $k$ such that $(s-1)^k F(s)$ is an entire function of order  $\leq 1$.\\\\
(3) Growth condition - The quantity $\frac{\mu_F(\sigma)}{(1-2\sigma)}$ is bounded for $\sigma<0$.\\\\
(4) Ramanujan Hypothesis - $|a_F(n)| = O_\epsilon(n^\epsilon)$ for any $\epsilon >0$.\\
\end{definition}

Define
\begin{equation*}
c_F := \limsup_{\sigma<0} \frac{2\mu_F({\sigma})}{1-2\sigma}, \hspace{3mm} c_F^* := \limsup_{\sigma<0} \frac{2\mu_F^*({\sigma})}{1-2\sigma}.
\end{equation*}
\vspace{1mm}

As we shall see, these will play the role of degree for elements of $\mathbb{M}$.

We further introduce a stronger growth condition on the functions in $\mathbb{M}$. The results in this paper will hold subject to these conditions.

\begin{definition}
{\bf The class $\mathbb{M}_1$}. A function $F\in\mathbb{M}$ is in $\mathbb{M}_1$ if it further satisfies the following. 

There exist $\epsilon, \delta > 0$ such that
    $$
    \limsup_{|s|\to\infty} \{ s : |\arg(s)|>\pi/2 + \delta, |F(s)| \gg_\sigma (|s| + 2)^{\mu_F^*(\sigma) - \epsilon}\} = \infty.
    $$
\end{definition}    

\section{Nevanlinna Theory}\label{nevanlinna}

Nevanlinna theory was introduced by R. Nevanlinna \cite{Nev} to study the value-distribution of meromorphic functions. We recall some basic definitions and facts commonly used.

Let $f$ be a meromorphic function and denote the number of poles of $f(s)$ in $|s|<r$ by $n(f,r)$ counting multiplicities and denote by $n(f,c,r)$ the number of $c$-values of $f$ in $|s|<r$, counting multiplicities. Indeed
\begin{equation*}
    n(f,c,r) = n\left(\frac{1}{f-c}, r\right).
\end{equation*}

The integrated counting function is defined as

\begin{equation}\label{ten}
    N(f,c,r) := \int_0^r \bigg( n(f,c,t) - n(f,c,0) \bigg) \frac{dt}{t} + n(f,c,0) \log r.
\end{equation}

and 

\begin{equation*}
    N(f,r) := \int_0^r \bigg(n(f,t) - n(f,0)\bigg) \frac{dt}{t} + n(f,0) \log r.
\end{equation*}

The proximity function is defined by

\begin{equation*}
    m(f,r) := \frac{1}{2\pi} \int_0^{2\pi} \log^+ |f(re^{i\theta})|d\theta,
\end{equation*}

and

\begin{equation*}
    m(f,c,r) := m\left(\frac{1}{f-c}, r\right),
\end{equation*}

where $\log^+ x = \max\{0,\log x\}$. The Nevanlinna characteristic function of $f$ is defined by

\begin{equation*}
    T(f,r) = N(f,r) + m(f,r).
\end{equation*}

We recall some basic properties of these functions.

\begin{enumerate}
    \item If $f$ and $g$ are meromorphic functions, then
    \begin{equation*}
    T(fg,r) \leq T(f,r) + T(g,r), \hspace{4mm} m(fg,r)\leq m(f,r) + m(g,r).
    \end{equation*}
    \begin{equation*}
    T(f+g,r) \leq T(f,r) + T(g,r) + O(1), \hspace{4mm} m(f+g,r)\leq m(f,r) + m(g,r) + O(1).
    \end{equation*}
    \vspace{2mm}
    
    \item The complex order of a meromorphic function is given by
    \begin{equation*}
        \rho(f) := \limsup_{r\to\infty} \frac{\log T(f,r)}{\log r}.
    \end{equation*}
    \vspace{2mm}
    
    \item If $\rho(f)$ is finite, then we have the logarithmic derivative lemma (see \cite{Nev}, p. 370),
    \begin{equation}\label{logl}
        m\left(\frac{f'}{f},r\right) = O(\log r).
    \end{equation}
\end{enumerate}
\vspace{2mm}

The main theorem of Nevanlinna states that $T(f,c,r):= N(f,c,r) + m(f,c,r)$ differs from the characteristic function by a bounded quantity.

\begin{theorem}\label{neva}
(First Fundamental Theorem) Let $f$ be a meromorphic function and let $c$ be any complex number. Then
\begin{equation*}
    T(f,c,r) = T(f,r) + O(1),
\end{equation*}
where the error term depends on $f$ and $c$ and is independent of $r$.
\end{theorem}

We also have a bound for the Nevanlinna characteristic in terms of the value distribution of three or more values. This is often called the Second Fundamental theorem of Nevanlinna theory.

\begin{theorem}\label{neva2}
(Second Fundamental Theorem) Let $f$ be a meromorphic function of finite order. Suppose $a_j$'s are distinct complex values (including $\infty$) and $q\geq 3$, we have
\begin{equation*}
    (q-2) T(f,r) \leq \sum_{j=1}^q \overline{N} \left(\frac{1}{f-a_j},r\right) + O(\log r),
\end{equation*}
where $\overline{N}$ is the integrated counting function defined similarly as $N$, but without counting multiplicity.
\end{theorem}

\section{The main results}\label{Main}

Let $\mathbb{M}_1$ denote the class of $L$-functions as defined in section 2 and $T(r,f)$ denote the Nevanlinna characteristic of a meromorphic function $f$(refer section 3). We establish the following results.
\begin{theorem}\label{thm1}
If $L \in \mathbb{M}_1$, then
\begin{equation*}
    T(L,r) = \Omega(r\log r).
\end{equation*}

\end{theorem}

\begin{theorem}\label{thm2}
If two $L$-functions $L_1, L_2\in\mathbb{M}_1$ share a complex value $c$ counting multiplicity, then $L_1 = aL_2 -ac +c$ for some $a\in\mathbb{C}$.
\end{theorem}

Denote $n(f,r)$ as the number of poles of $f$ counting multiplicity in $|s|<r$. We say that $f$ and $g$ share a complex value $c$ up to an error term $E(r)$, if 
\begin{equation*}
n\left(\frac{1}{(f-c)} - \frac{1}{(g-c)},r\right) \leq E(r).
\end{equation*}

Similarly, denote $\overline{n}(f,r)$ as the number of poles of $f$ ignoring multiplicity in $|s|<r$. We say that $f$ and $g$ share a complex value $c$ ignoring multiplicity, up to an error term $E(r)$, if 
\begin{equation*}
    \overline{n}\left(\frac{1}{(f-c)} - \frac{1}{(g-c)},r\right) \leq E(r).
\end{equation*}

\begin{theorem}\label{thm3}
Let $f$ be any meromorphic function on $\mathbb{C}$ with finitely many poles and $L\in\mathbb{M}_1$ be such that they share one complex value counting multiplicity and another complex value ignoring multiplicity, up to an error term $o(r\log r)$, then $f\equiv L$.
\end{theorem}
\begin{theorem}\label{thm4}
Let $f$ be any meromorphic function on $\mathbb{C}$  of order $\leq 1$, with finitely many poles and $L\in\mathbb{M}_1$ be such that they share one complex value counting multiplicity and their derivatives, $f'$ and $L'$ share zeroes up to an error term $o(r\log r)$,  then $f= \mu L$, where $\mu$ is a root of unity.
\end{theorem}

Since $\mathbb{M}_1$ contains the Selberg class, we have, in particular, the following corollary.

\begin{corollary}
Let $f$ be a meromorphic function on $\mathbb{C}$ of order $\leq 1$ with finitely many poles and $L$ be an $L$-function in the Selberg class, such that they share one complex value counting multiplicity and their derivatives, $f'$ and $L'$ share zeroes up to an error term $o(r\log r)$,  then $f\equiv L$.
\end{corollary}

\begin{remark}\label{rem1}
We also prove Theorem \ref{thm1} without the assumption of the above stronger growth condition if one of the following holds.
\begin{enumerate}
    \item {\bf Distribution on vertical lines}: There exist $\sigma<0$ and $\epsilon>0$ such that
    $$
    \lim_{T\to\infty} \frac{1}{T} \hspace{2mm} meas \{|t|<T : |F(\sigma + it)| \gg T^{\epsilon} \} >0.
    $$
    \item {\bf Equidistribution of zeroes}: Let $N(L,0,T) = \Omega(f(T))$ for $T\gg 1$, then
    $$
    N(L,0,T+1) - N(L,0,T) = O(f(T)/T),
    $$
    where $N(L,c,T)$ is defined as \eqref{ten}.
    
\end{enumerate}
\end{remark}

Note that the Selberg class satisfies all the above conditions including the stronger growth condition because of the functional equation. We expect $\mathbb{M}_1 = \mathbb{M}$, but we do not have a proof.

\section{Preliminaries}\label{preliminaries}

Using Jensen's theorem, we know that
\begin{equation*}
    N(L,0,r) := \int_{0}^r \frac{n(L,0,t)}{t}dt =  \frac{1}{2\pi} \int_0^{2\pi} \log |L(\sigma_0 +1 +iT + re^{i\theta})|d\theta - \log |L(\sigma_0 + 1 + iT|,
\end{equation*}
which implies
\begin{equation*}
    N(L,0,r) = O(r\log r).
\end{equation*}

We show that for $F\in\mathbb{M}_1$ with $c_F^*>0$, the number of zeroes in the disc of radius $r$ is in fact $\Omega(r\log r)$.

\begin{prop}\label{prop5.1}
Let $F\in\mathbb{M}_1$ and $c_F^*>0$, then
\begin{equation*}
    N(F,0,r) = \Omega(r\log r)
\end{equation*}
\end{prop}
\begin{proof}
If $F$ has a pole of order $k$ at $s=1$, we define
\begin{equation*}
    G(s):= (s-1)^k F(s).
\end{equation*}
Note that $G(s)$ is entire and also satisfies the growth conditions of $F$. By Hadamard product factorization, we have
\begin{equation*}
    G(s) = s^m e^{As+B} \prod_{\rho} \left( 1 - \frac{s}{\rho} \right) e^{s/\rho},
\end{equation*}
where $\rho$ runs over the zeros of $G$ and $m,A,B$ are constants. We use the following result (see \cite{Gol}, p. 56, Remark 1), which states that if 
\begin{equation*}
    H(s) = \prod_{\rho} \left( 1 - \frac{s}{\rho} \right) e^{s/\rho}
\end{equation*}

and $\sum_{\rho} 1/|\rho|^2$ is bounded, then 
\begin{equation}\label{bound}
    \log |H(s)| \ll |s| \int_0^{|s|} \frac{n(t,0,H)}{t^2} dt + |s|^2\int_{|s|^2}^{\infty} \frac{n(t,0,H)}{t^3} dt.
\end{equation}
Recall that, by Jensen's theorem we have $N(T+1,0,G)- N(T,0,G) = O(\log T)$. Hence,
\begin{equation*}
    \sum_{G(\rho)=0} \frac{1}{|\rho|^2}
\end{equation*}
is bounded. Applying \eqref{bound} to $G$, we have
\begin{equation}\label{two}
    \log |G(s)| \ll |s| \int_0^{|s|} \frac{n(t,0,G)}{t^2} dt + |s|^2\int_{|s|^2}^{\infty} \frac{n(t,0,G)}{t^3} dt + O(|s|).
\end{equation}
If we assume that $N(T,0,F) = o(T\log T)$, then we show that RHS of \eqref{two} is $o(T\log T)$. To see this, we first show that if $N(T,0,G)=o(T\log T)$, then $n(T,0,G) = o(T\log T)$. Suppose $n(T,0,G)$ is not $o(T\log T)$, then there exists infinitely many $T$ such that $n(T,0,G)\ll T\log T$. But, that implies
\begin{equation*}
    N(2T,0,G)=\int_{0}^{2T}\frac{n(t,0,G)}{t} dt\geq \int_{T}^{2T} \frac{n(t,0,G)}{t} dt \gg T\log T,
\end{equation*}
which contradicts the assumption that $N(T,0,F) = o(T\log T)$. Therefore, it follows that the RHS of \eqref{two} is $o(T\log T)$.

But the strong growth condition implies that we can find $s$, with $|s|$ arbitrarily large such that the LHS of \eqref{two} is $\Omega(T\log T)$. This leads to a contradiction.
\end{proof}
We now show that the above proposition can be realized by dropping the strong growth condition under some distribution assumptions (see Remark \ref{rem1}). 
\begin{prop}\label{prop5.2}
If $F\in\mathbb{M}$ satisfies one of the following conditions,
\begin{enumerate}
    \item there exist $\sigma<0$ and $\epsilon>0$ such that
    $$
    \lim_{T\to\infty} \frac{1}{T} \hspace{2mm} meas \{|t|<T : |F(\sigma + it)| \gg T^{\epsilon} \} >0, \text{ or}
    $$
    \item if $N(L,0,T) = \Omega(f(T))$ for $T\gg 1$, then
    $$
    N(L,0,T+1) - N(L,0,T) = O(f(T)/T),
    $$
    then $N(T,0,F) = \Omega(T\log T)$.
\end{enumerate}
\end{prop}
\begin{proof}
We invoke the following theorem of Landau (see \cite{Tit}, p. 56, sec. 3.9, Lemma $\alpha$)
\begin{theorem*}
If $f$ is holomorphic in $|s-s_0|\leq r$ and $|{f(s)}/{f(s_0)}| < e^M$ in $|s-s_0|\leq r$, then
\begin{equation}\label{lan}
    \left|\frac{f'}{f}(s) - \sum_{|s_0-\rho|<r/2} \frac{1}{(s-\rho)} \right| \leq \frac{4}{1-2\alpha^2} \frac{M}{r},
\end{equation}
where $|s-s_0| \leq \alpha r$ for any $\alpha<1/2$ and $\rho$ runs over the zeroes of $f$.
\end{theorem*}

If $L\in\mathbb{M}$, then from Jensen's theorem, we have
\begin{equation*}
    \frac{L'}{L} (s) = \sum_{|s_0-\rho|<r/2} \frac{1}{(s-\rho)} + O(\log t),
\end{equation*}
where $s=\sigma + it$.

Let $C_1, C_2, C_3$ be circles with center $2+iT$ and radius $2-2\sigma_0, 2-\sigma_0$ and $1/2$ respectively, where $\sigma_0<0$. From \eqref{lan}, we have for $s\in C_2$
\begin{equation*}
    \frac{L'}{L} (s) = \sum_{\rho\in C_1}  \frac{1}{(s-\rho)} + O(\log T).
\end{equation*}

Moreover, since $L$ has a Dirichlet series on $Re(s)>1$, we have that it is bounded above and below on $C_3$ and hence
\begin{equation*}
    \frac{L'}{L} (s) = O(1).
\end{equation*}

Define the function 
\begin{equation*}
    g(s) = L(s) \prod_{\rho \in C_1} \frac{1}{s-\rho}.
\end{equation*}

If $N(L,0,T) = o(T\log T)$, then the condition(2) implies that we have for $s\in C_3$
\begin{equation*}
    \left|\frac{g'}{g} (s) \right| = o(\log T).
\end{equation*}

Moreover, for $s\in C_2$, we have $\left| g'(s)/g(s) \right| = O(\log T)$. By Hadamard's three-circle theorem, we have for any circle $C_4$ with center $2+iT$ and radius $2-\sigma_0 + \delta$,
\begin{equation}\label{Hada}
    \frac{g'}{g} (s) \ll  (\log T)^{\alpha_1} o\left((\log T)^{\alpha_2}\right) = o(\log T),
\end{equation}
where $s\in C_4$, $\delta>0$, $0<\alpha_1,\alpha_2<1$ and $\alpha_1 + \alpha_2 =1$.

Now, consider the integral
\begin{multline}\label{one}
    \int_{\sigma_0 + 3\delta +iT}^{2+iT}\frac{g'}{g}(s) ds = \log L(2+iT) - \log L(\sigma_0 + 3\delta +iT) \\
    - \sum_{\rho\in C_1} \log L(2+iT-\rho) - \log L(\sigma_0 + 3\delta +iT-\rho).
\end{multline}

By \eqref{Hada}, LHS of \eqref{one} is $o(\log T)$. But, by the growth condition we can choose $T$ such that
\begin{equation*}
    \log |L(\sigma_0 + iT)| = \Omega(\log T).
\end{equation*}

Thus, the RHS of \eqref{one} is $\Omega (\log T)$, because all the terms except $\log |L(\sigma + iT)|$ is $o(\log T)$. This is a contradiction.

Instead, if we assume condition(1) and suppose that $N(L,0,T) = o(T\log T)$. We have
\begin{equation*}
    \lim_{T\to\infty} \frac{1}{T} \hspace{2mm} meas \{|t|<T : N(L,0,T+1) - N(L,0,T) = \Omega(\log T) \} =0.
\end{equation*}
But condition(1) implies that we can find $T$ and $\sigma_0$ such that $N(L,0,T+1) - N(L,0,T) = o(\log T)$ and $\log |L(\sigma_0 + iT)| = \Omega(\log T)$. Now, we follow the same argument as before.
\end{proof}

\section{Proof of the theorems}\label{proofs}
\subsection{Proof of Theorem \ref{thm1}}
\begin{proof}
We evaluate the Nevanlinna characteristics for $L$-functions in $\mathbb{M}$. For $L\in \mathbb{M}$, we have
\begin{equation*}
    m(L,r) = \frac{1}{2\pi} \int_0^{2\pi} \log^+ |L(re^{i\theta})| d\theta.
\end{equation*}

Since, $L(s)$ is bounded on $\sigma>1$, we have
\begin{equation*}
    \frac{1}{2\pi} \int_{\theta; r \cos\theta>1} \log^+ |L(re^{i\theta})| d\theta \ll 1.
\end{equation*}
Moreover, using the growth condition, we have
\begin{align*}
    \frac{1}{2\pi} \int_{\theta; r \cos\theta<1} \log^+ |L(re^{i\theta})| d\theta & \leq \frac{1}{2\pi} \int_{\theta; r \cos\theta<1} \mu_L(r\cos\theta) d\theta + O(r) \\
    & \leq  \frac{c_L}{\pi} r\log r + O(r).
\end{align*}
Thus, we conclude that
\begin{equation*}
    m(L,r) \leq \frac{c_L}{\pi} r\log r + O(r).
\end{equation*}
Moreover, since $L\in\mathbb{M}$ has only one possible pole at $s=1$, we have for $r>1$,
\begin{equation*}
    N(L,r) \leq \log r.
\end{equation*}
Hence,
\begin{equation*}
    T(L,r) \leq \frac{c_L}{\pi} r\log r + O(r).
\end{equation*}
In order to prove Theorem \ref{thm1}, we use Theorem \ref{neva}, namely,
\begin{equation*}
    T(L,0,r) = T(L,r) + O(1).
\end{equation*}
But, we also know that
\begin{equation*}
    N(L,0,r) \leq T(L,0,r).
\end{equation*}
Hence, from Proposition \ref{prop5.1} and \ref{prop5.2}, we have $N(L,0,r) = \Omega(r\log r)$. Therefore,
\begin{equation*}
    T(L,r) = \Omega( r\log r).
\end{equation*}
This proves Theorem \ref{thm1}.
\end{proof}
Note that, the best known result for the number of zeroes of a general Dirichlet series $F(s)$ with meromorphic continuation is due to Bombieri and Perelli \cite{Bom}, which states that
\begin{equation*}
    \lim\sup_{T\to\infty} \frac{N(T,0,L) + N(T,\infty,F)}{T^\delta} > 0,
\end{equation*}
for $\delta<1$. In our case, enforcing a stronger growth condition ensures the number of zeroes to be $\Omega(T\log T)$.

\subsection{Proof of Theorem \ref{thm2}}
\begin{proof}
Suppose $L_1, L_2 \in \mathbb{M}$ share one complex value $c$, CM. Since $L_1$ and $L_2$ have only one possible pole at $s=1$, we define $F$ as
\begin{equation*}
F:= \frac{L_1 - c}{L_2 - c} Q,
\end{equation*}
where $Q = (s-1)^k$ is a rational function such that $F$ has no poles or zeroes. Since, $L_1$ and $L_2$ have complex order $1$, we conclude that $F$ has order at most $1$ and hence is of the form
\begin{equation*}
    F(s) = e^{as +b}.
\end{equation*}
This immediately leads to $a=0$, since $L_1$ and $L_2$ are absolutely convergent on $Re(s)>1$ and taking $s\to\infty$, $L_1(s)$ and $L_2(s)$ approach their leading coefficient. This forces 
\begin{equation*}
    L_1(s) = a L_2(s) + b.
\end{equation*}
for some constants $a,b\in\mathbb{C}$. Moreover, since they share a $c$-value, $b= c-ac$.
\end{proof}

\subsection{Proof of Theorem \ref{thm3}}
\begin{proof}
We argue similarly as in \cite{Bao1}. Suppose $L\in\mathbb{M}$ and a meromorphic function $f$ share a complex value $a$ CM and another complex value $b$ IM with an error term up to $o(r\log r)$. Consider the auxiliary function
\begin{equation}\label{aux}
G:= \left(\frac{L'}{(L-a)(L-b)} - \frac{f'}{(f-a)(f-b)}\right) (f-L).
\end{equation}

We first claim that $N(r,G) = o(r\log r)$. The only poles of the function $G$ comes from the zeroes of denominators in \eqref{aux} and the poles of $f$. 

For any zero $z$ of $L-a$ and $f-a$, $L'/(L-a)(L-b)$ and $f'/(f-a)(f-b)$ have the same principal part in the Laurent expansion at $s=z$, because $L$ and $f$ share the value $a$ CM. Hence, every zero of $L-a$ is also a zero of $G$.

For zeroes of $L-b$ and $f-b$ in $|s|<r$, except for $o(r\log r)$ of them, $L'/(L-b)$ and $f'/(f-b)$ have a simple pole at those points which cancel with the zero of $(f-L)$. Thus, there are at most $o(r\log r)$ poles of $G$ in $|s|<r$ coming from the zeroes of $L-b$ and $f-b$.

Since $f$ has finitely many poles, we conclude that 
\begin{equation*}
    N(G,r) = o(r\log r).
\end{equation*}

Moreover, since $L-a$ and $f-a$ share zeroes with multiplicity, we have an entire function which neither has a pole nor a zero given by
\begin{equation*}
    F:= \frac{L-a}{f-a} Q,
\end{equation*}
where $Q$ is a rational function such that it cancels the poles of $f$. Hence, we have
\begin{equation*}
    F(s) = e^{g(s)}.
\end{equation*}
We prove that $g$ is at most a linear function. By Theorem \ref{neva2}, we have
\begin{align*}
    T(f,r) & < \overline{N}\left( \frac{1}{f-a},r\right) + \overline{N}\left( \frac{1}{f-b},r\right) + \overline{N}(f,r) + O(\log r)\\
    & = \overline{N}\left( \frac{1}{L-a},r\right) + \overline{N}\left( \frac{1}{L-b},r\right) + \overline{N}(f,r) + o(r\log r) \\
    & = O(r\log r).
\end{align*}

Hence, the complex order of $f$, given by
\begin{equation*}
    \rho(f) = \limsup_{r\to\infty} \frac{\log T(f,r)}{\log r} \leq 1. 
\end{equation*}

Thus, $f$ is of order at most $1$ and since $L$ is of order $1$, we conclude that $g$ is linear.\\

In order to prove Theorem \ref{thm3}, it suffices to show that $G\equiv 0$. We establish this by computing the Nevanlinna characteristic of $G$.

Since $g(s)$ is linear and $Q$ a rational function, $T(F,r) = O(r)$ and $T(Q,r) = O(\log r)$. We thus have
\begin{equation*}
    m\left(\frac{f-L}{L-a}, r\right) \leq T\left( \frac{f-L}{L-a}, r\right) = T\left( \frac{Q}{F}-1, r\right) \leq O(r).
\end{equation*}
Similarly,
\begin{equation*}
    m\left(\frac{f-L}{f-a}, r\right) \leq T\left( \frac{f-L}{f-a},r\right) = T\left(1-\frac{F}{Q}, r\right) \leq O(r).
\end{equation*}
Using the logarithmic derivative lemma \eqref{logl}, we have
\begin{equation*}
    m\left(\frac{f'}{f-b}, r\right) = O(\log r)\text{  and  }m\left( \frac{L'}{L-b}, r\right) = O(\log r).
\end{equation*}
Therefore, we conclude
\begin{equation*}
    m(G,r) = O(r).
\end{equation*}
Since, $N(G,r) = o(r\log r)$, we get
\begin{equation*}
    T(G,r) = o(r\log r).
\end{equation*}
From Theorem \ref{neva}, we have
\begin{equation*}
    T\bigg(\frac{1}{G}, r \bigg) = T(G, r) + O(1) = o(r\log r).
\end{equation*}
But, note that every zero of $G$ is also a zero of $L-a$. Therefore, 
\begin{equation*}
    N\bigg(\frac{1}{G},r \bigg) \geq N(L-a, 0, r) = \Omega(r\log r). 
\end{equation*}
This is a contradiction since
\begin{equation*}
    N\bigg(\frac{1}{G},r \bigg)\leq T\bigg(\frac{1}{G},r\bigg) = o(r\log r).
\end{equation*}
\end{proof}

\subsection{Proof of Theorem \ref{thm4}}
\begin{proof}
Suppose $f$ is a meromorphic function on $\mathbb{C}$ of order $\leq 1$, with finitely many poles and $L\in\mathbb{M}$ such that they share complex value $a$ counting multiplicity and their derivatives $f'$ and $L'$ share zeroes up to an error term $o(r\log r)$.

Since $L-a$ and $f-a$ share zeroes with multiplicity, we have an entire function which neither has zeroes nor poles given by
\begin{equation*}
    F := \frac{L-a}{f-a} Q,
\end{equation*}
where $Q$ is a rational function such that it cancels the poles of $f$. Hence, we have
\begin{equation*}
    F(s) = e^{g(s)}.
\end{equation*}
Since $f$ has complex order $\leq 1$, we get $g(s)$ is linear.\\

Consider the auxiliary function
\begin{equation}\label{aux2}
G(s) := \left( \frac{1}{(L-a)f'} - \frac{1}{(f-a)L'}\right)(f'-L')(f-L).
\end{equation}

Now, we do a similar analysis as in the proof of Theorem \ref{thm3}. We first claim that $N(G,r) = o(r\log r)$. The only poles of $G$ can arise from the zeroes of the denominator in \eqref{aux2}.

For any zero $z$ of $L-a$ and $f-a$, we have $L'/(L-a) L'f'$ and $f'/(f-a)f'L'$ have the same principal part in the Laurent expansion at $s=z$, because $L$ and $f$ share the value $a$ CM. Hence, every zero of $L-a$ is also a zero of $G$.

For zeroes of $L'$ and $f'$ in $|s|<r$, except for $o(r\log r)$ of them, they are also zeroes $f'-L'$ of same multiplicity. Thus, there are at most $o(r\log r)$ poles of $G$ in $|s|<r$ coming from the zeroes of $L'$ and $f'$.

Since, $f$ has finitely many poles, so does $f'$ and hence we conclude
\begin{equation*}
    N(G,r) = o(r\log r).
\end{equation*}

In order to prove Theorem \ref{thm4}, it suffices to show that $G\equiv 0$. We establish this by computing the Nevanlinna characteristic of $G$.

Since $g(s)$ is linear and $Q$ a rational function, $T(F,r) = O(r)$ and $T(Q,r) = O(\log r)$. We thus have
\begin{equation*}
    m\bigg(\frac{f-L}{L-a},r\bigg) \leq T\bigg( \frac{f-L}{L-a},r\bigg) = T\bigg( \frac{Q}{F}-1,r\bigg) \leq O(r).
\end{equation*}
Similarly,
\begin{equation*}
    m\bigg(\frac{f-L}{f-a},r\bigg) \leq T\bigg( \frac{f-L}{f-a},r\bigg) = T\bigg( 1-\frac{F}{Q},r\bigg) \leq O(r).
\end{equation*}

Note that
\begin{equation*}
    \frac{f'-L'}{f'} = 1- \frac{L'}{f'}.
\end{equation*}
Since $L(s) -a = (f(s)-a) \frac{e^{g(s)}}{Q}$, taking derivatives we get
\begin{equation*}
    L'(s) = f'(s) \frac{e^{g(s)}}{Q} + (f(s)-a) \left(\frac{e^{g(s)}}{Q}\right)'.
\end{equation*}
Dividing by $f'(s)$, we have
\begin{equation*}
    \frac{L'}{f'}(s) = \frac{e^{g(s)}}{Q} + \frac{(f(s)-a)}{f'} \left(\frac{e^{g(s)}}{Q}\right)'.
\end{equation*}
We calculate the proximity function for the right hand side.
\begin{equation*}
    m\bigg(\frac{e^{g(s)}}{Q},r\bigg) \leq m( e^{g(s)},r) + m(Q,r) = O(r),
\end{equation*}
\begin{equation*}
    m\left(\frac{(f(s)-a)}{f'} \left(\frac{e^{g(s)}}{Q}\right)',r\right) \leq m\left(\frac{(f(s)-a)}{f'},r\right) + m\left( \left(\frac{e^{g(s)}}{Q}\right)',r\right).
\end{equation*}
Using logarithmic derivative lemma \eqref{logl}, we get
\begin{equation*}
    m\left(\frac{(f(s)-a)}{f'} \left(\frac{e^{g(s)}}{Q}\right)',r\right) = O(r).
\end{equation*}

Therefore, we have
\begin{equation*}
    m\left(\frac{f'-L'}{f'},r\right) = O(r).
\end{equation*}

Similarly, we also get

\begin{equation*}
    m\left(\frac{f'-L'}{L'},r\right) = O(r).
\end{equation*}

Hence, we conclude
\begin{equation*}
    T(G,r) = o(r\log r).
\end{equation*}
Now, we again proceed as in proof of Theorem \ref{thm3}. By Theorem \ref{neva}, we have

\begin{equation*}
    T\bigg(\frac{1}{G},r \bigg) = T(G,r) + O(1).
\end{equation*}

Moreover, every zero of $L-a$ is also a zero of $G$. Hence,

\begin{equation*}
N\bigg(\frac{1}{G},r \bigg)\geq N(L-a,0,r) = \Omega(r\log r).
\end{equation*}
This contradicts the fact that
\begin{equation*}
    N\bigg(\frac{1}{G},r \bigg) \leq T\bigg( \frac{1}{G},r \bigg) = o(r\log r).
\end{equation*}

\end{proof}

\section{Acknowledgements}
I would like to extend my gratitude to Prof. V. Kumar Murty for his guidance and insightful comments.

\end{document}